\documentclass[12pt]{article}
\usepackage{amssymb}
\usepackage{amsmath}
\usepackage{amsfonts}

\newtheorem{theorem}{Theorem}

\newtheorem{lemma}[theorem]{Lemma}

\newtheorem{remark}[theorem]{Remark}

\newenvironment{proof}[1][Proof]{\noindent\textbf{#1.} }{\ \rule{0.5em}{0.5em}}
\input{tcilatex}
\begin{document}

\title{The Dirichlet problem for the minimal hypersurface equation with
Lipschitz continuous boundary data on domains of a Riemannian manifold }
\author{A. Aiolfi, G. Nunes, L. Sauer, R. B. Soares.}
\date{}
\maketitle

\begin{abstract}
Given a $C^{2}$-domain $\Omega \subset M$ with compact boundary, where $M$
is an arbitrary complete Riemannian manifold, we search for smallness
conditions on the boundary data for which the Dirichlet problem for the
minimal hypersurface equation is solvable. We obtain an extension to
Riemannian manifolds of an existence result of G. H. Williams ( J. Reine
Angew. Math. 354:123-140, 1984).
\end{abstract}

\section{Introduction}

\qquad Let $M^{n}$, $n\geq 2$, be a complete Riemannian manifold and let $%
\Omega \subset M$ be a $C^{2}$-domain with compact boundary. We consider the
Dirichlet problem%
\begin{equation}
\left\{ 
\begin{array}{c}
\mathfrak{M}\left( u\right) :=\func{div}\left( \frac{\func{grad}u}{\sqrt{%
1+\left\vert \func{grad}u\right\vert ^{2}}}\right) =0\text{ in }\Omega \text{%
, }u\in C^{2}\left( \Omega \right) \cap C^{0}\left( \overline{\Omega }\right)
\\ 
u|_{\partial \Omega }=f%
\end{array}%
\right.  \label{PD}
\end{equation}%
where $f\in C^{0}\left( \partial \Omega \right) $ is given \textit{a priori}%
, $\func{grad}$ and $\func{div}$ are the gradient and divergence in $M$. If $%
u$ is a solution of (\ref{PD}) then the graph of $u$ is a minimal
hypersurface of $M\times \mathbb{R}$.

When $M=\mathbb{R}^{n}$ and $\Omega $ is bounded, it is well known that
problem (\ref{PD}) is solvable for any $f\in C^{0}\left( \partial \Omega
\right) $ if and only if $\Omega $ is mean convex (Theorem 1 of \cite{JS}).
The existence part of this result has been extended and generalized to the
Riemannnian setting (see for instance \cite{DHL},\cite{ET},\cite{NR}, \cite%
{ARS}). If the boundary data $f$ is restricted in some way, problem (\ref{PD}%
) may be solvable even if $\Omega $ is non mean convex. It was shown in
Theorem 1 of \cite{ARS} - which is an extension to Riemannian manifold of
the classical result of H. Jenkins and J. Serrin (Theorem 2 of \cite{JS}) -
that if $f\in C^{2}\left( \partial \Omega \right) $ and 
\begin{equation}
osc\left( f\right) =\underset{\partial \Omega }{\sup }f-\underset{\partial
\Omega }{\inf }f\leq \mathfrak{C}\left( \left\vert Df\right\vert ,\left\vert
D^{2}f\right\vert ,\left\vert A\right\vert ,Ric_{M}\right) ,  \label{cso}
\end{equation}%
where $\left\vert A\right\vert $ denotes the norm of the second fundamental
for of $\partial \Omega $ and $\mathfrak{C}$ is a function which has an
explicit form (Section 2, p. 78 of \cite{ARS}), then problem (\ref{PD}) is
solvable. However, at least in Euclidean spaces, such smallness condition on
the boundary data is not the least. In fact, it seems that the least
restrictive condition was given by G. H. Williams in Theorem 1 of \cite{W}.
He shows that for a non mean convex bounded $C^{2}$-domain $\Omega \subset 
\mathbb{R}^{n}$ and $f\in C^{0,1}\left( \partial \Omega \right) $ with
Lipschitz constant%
\begin{equation*}
Lip\left( f\right) =K\in \lbrack 0,\frac{1}{\sqrt{n-1}}),
\end{equation*}%
the problem (\ref{PD}) is solvable if $osc\left( f\right) $ is sufficiently
small, $osc\left( f\right) <\varepsilon \left( n,K,\Omega \right) $
(Corollary 1 of \cite{W}). He also shows that if $K>\frac{1}{\sqrt{n-1}}$
then\textbf{\ }there is a positive boundary data $f$ with $Lip\left(
f\right) =K$ such that (\ref{PD}) has no classical solution (Theorem 4 of 
\cite{W}).

Williams' results (existence and non existence) were extended to unbounded
domain $\Omega \subset \mathbb{R}^{2}$ by N. Kutev and F. Tomi in \cite{KT}
and J. Ripoll and F. Tomi gave, in the specific case $\Omega \subset \mathbb{%
R}^{2}$, Williams' condition in a more explicit form (see Theorem 1 of \cite%
{RT}). Schulz and Williams \cite{SW} and Bergner \cite{B} generalized
Williams' result to prescribed mean curvature (in Euclidean spaces). Here,
our main objective is to obtain an extension of Williams' ex\nolinebreak
istence theorem \cite{W} to Riemannian manifolds. In order to state our main
result, we establish some notation.

Let $\nu $ be the unit normal vector field to $\partial \Omega $ which point
to $\Omega $. Let $\mathfrak{H}$ be the mean curvature of $\partial \Omega $
with respect to $\nu $ and set%
\begin{equation}
\partial ^{-}\Omega =\limfunc{clos}\left\{ x\in \partial \Omega ;\mathfrak{H}%
\left( x\right) <0\right\}  \label{nmc}
\end{equation}%
and 
\begin{equation}
\mathfrak{H}_{\inf }=\inf \left\{ \mathfrak{H}\left( x\right) ;x\in \partial
^{-}\Omega \right\} .  \label{infb}
\end{equation}

Given $x\in \partial ^{-}\Omega $, let $R\left( x\right) $ the maximal
radius of the normal sphere contained in $M-\Omega $ which is tangent to $%
\partial \Omega $ at $x$ and set 
\begin{equation}
R=\inf \left\{ R\left( x\right) ;x\in \partial ^{-}\Omega \right\} .
\label{R}
\end{equation}%
Since $\partial ^{-}\Omega $ is compact, $R>0$.

Let $r\in (0,R)$. Given $x\in \partial ^{-}\Omega $ set 
\begin{equation}
x^{\ast }=\exp _{x}r\left( -\nu \left( x\right) \right) ,  \label{x*}
\end{equation}%
$B_{r}\left( x^{\ast }\right) $ the normal ball with center at $x^{\ast }$
and radius $r$ and consider the normal sphere $\Sigma =\partial B_{r}\left(
x^{\ast }\right) $. Let $\eta $ be the unit normal vector field to $\Sigma $
which points to $M-B_{r}\left( x^{\ast }\right) $. Let $\lambda _{i}$, $%
i=1,...,n-1$, be the principal curvatures of $\Sigma $ with respect to $\eta 
$ and set 
\begin{equation}
\lambda \left( x\right) =\min \left\{ \lambda _{i}\left( p\right) ,\text{ }%
i=1,...,n-1,p\in \Sigma \right\} .  \label{lambdax}
\end{equation}%
Let $II_{\partial \Omega }$ be the second fundamental form of $\partial
\Omega $ relatively to $\nu $ and let\textbf{\ }$II_{\Sigma }$ be the second
fundamental form of $\Sigma $ relatively to $\eta $. Notice that, at $x$, $%
\nu \left( x\right) =\eta \left( x\right) $, since $T_{x}\partial \Omega
=T_{x}\Sigma $. Set%
\begin{equation}
\varkappa \left( x\right) =\min \left\{ II_{\partial \Omega }\left( v\right)
-II_{\Sigma }\left( v\right) ;\text{ }v\in T_{x}\partial \Omega \text{, }%
\left\vert v\right\vert =1\right\} .  \label{vkapx}
\end{equation}

Now, consider the real numbers 
\begin{equation}
\lambda _{r}:=\inf \left\{ \lambda \left( x\right) ;\text{ }x\in \partial
^{-}\Omega \right\}  \label{lambda}
\end{equation}%
and 
\begin{equation}
\varkappa _{r}:=\inf \left\{ \varkappa \left( x\right) ;\text{ }x\in
\partial ^{-}\Omega \right\} .  \label{vkapr}
\end{equation}%
Notice that $\lambda _{r}<0$ and, since $r<R$, we have $\varkappa _{r}>0$.

Finally, consider $\varrho >0$ as the biggest number such that \newline
$\exp _{\partial \Omega }:\partial \Omega \times \lbrack 0,\varrho
)\longrightarrow \overline{\Omega }$ is a diffeomorphism and set%
\begin{equation}
\Omega _{\varrho }=\exp _{\partial \Omega }\left( \partial \Omega \times
\lbrack 0,\varrho )\right)  \label{ovr}
\end{equation}

We obtain the following extension of Williams' existence result.

\newpage

\begin{theorem}
\label{MT}Let $M^{n}$, $n\geq 2$, be a complete Riemannian manifold, $\Omega
\subset M$ be a $C^{2}$-domain with compact boundary and assume that $\func{%
Ric}_{\Omega _{\varrho }}\leq 0$.\newline
i) If $\partial ^{-}\Omega =\emptyset $, then the Dirichlet problem (\ref{PD}%
) has a bounded solution for any $f\in C^{0}\left( \partial \Omega \right) $.%
\newline
ii) If $\partial ^{-}\Omega \neq \emptyset $, assume additionally that $%
\func{Ric}_{\Omega _{\varrho }}\geq -\left( n-1\right) \mathfrak{H}_{\inf
}^{2}$. Then, given $r\in \left( 0,R\right) $, $a\in (0,\sqrt{\frac{r}{%
\left( n-1\right) \left\vert \lambda _{r}\right\vert }})$ and $K\in \lbrack
0,a\sqrt{\varkappa _{r}/r})$, where $R$, $\lambda _{r}$ and $\varkappa _{r}$
are given by (\ref{R}), (\ref{lambda}) and (\ref{vkapr}) respectively, there
is $0<\delta _{0}$ such that, for all $\delta \in (0,\delta _{0})$ there is $%
\epsilon =\epsilon \left( r,a,K,\delta ,\Omega ,\func{Ric}_{\Omega _{\varrho
}}\right) >0$ such that, if $f\in C^{0}\left( \partial \Omega \right) $
satisfies%
\begin{equation*}
\left\vert f\left( z\right) -f\left( x\right) \right\vert \leq Kd\left(
z,x\right) \text{, }x\in \partial ^{-}\Omega \text{, }z\in B_{\delta }\left(
x\right) \cap \partial \Omega
\end{equation*}%
and $osc\left( f\right) <\epsilon $, then the Dirichlet problem (\ref{PD})
has a bounded solution.\newline
Moreover, if $\Omega $ is bounded, the solutions mentioned in the items i)
and ii) are unique.
\end{theorem}

\begin{remark}
For $r\in \left( 0,R\right) $ and $a\in (0,\sqrt{\frac{r}{\left( n-1\right)
\left\vert \lambda _{r}\right\vert }})$ we have $a\sqrt{\varkappa _{r}/r}<%
\frac{1}{\sqrt{n-1}}$ (see Lemma \ref{m1}). When $M=\mathbb{R}^{n}$ we have $%
\left\vert \lambda _{r}\right\vert =1/r$ and, therefore, 
\begin{equation*}
\sqrt{\frac{r}{\left( n-1\right) \left\vert \lambda _{r}\right\vert }}=\frac{%
r}{\sqrt{n-1}}\text{ and }\frac{1}{r}\sqrt{\frac{R-r}{R}}\leq \sqrt{%
\varkappa _{r}/r}.
\end{equation*}%
It follows that, given $K\in \left( 0,\frac{1}{\sqrt{n-1}}\right) $ there
are $r\in \left( 0,R\right) $ and $a<\frac{r}{\sqrt{n-1}}$ such that $K<a%
\sqrt{\varkappa _{r}/r}$. Therefore, Theorem \ref{MT} extends to Riemannian
manifolds the existence results of G. H. Williams given in the Corollary 1
of \cite{W}.
\end{remark}

\begin{remark}
If the domain is strictly mean convex, the hypothesis on the Ricci curvature
is not necessary.
\end{remark}

\section{Barriers}

Consider the set 
\begin{equation}
S=\left\{ 
\begin{array}{c}
v\in C^{0}\left( \overline{\Omega }\right) ;v\text{ is subsolution of }%
\mathfrak{M}\text{,} \\ 
v\left( z\right) \leq f\left( z\right) \text{ }\forall z\in \partial \Omega 
\text{, }\sup_{\Omega }v\leq \sup_{\partial \Omega }f%
\end{array}%
\right\} .  \label{S}
\end{equation}%
Note that $S\neq \emptyset $ ($v=\inf_{\partial \Omega }f\in S$) and that
any function of\textbf{\ }$S$ is bounded from above by $w=\sup_{\partial
\Omega }f$. The function 
\begin{equation}
u\left( z\right) =\sup \left\{ v\left( z\right) ;v\in S\right\} \text{, }%
z\in \Omega  \label{Ps}
\end{equation}%
is then well defined. From Perron's method it follows that $u\in C^{2}\left(
\Omega \right) $ and $\mathfrak{M}\left( u\right) =0$ (see \cite{GT} and
Section 2 of \cite{RTe}). We will prove that $u\in C^{0}\left( \overline{%
\Omega }\right) $ and $u|_{\partial \Omega }=f$. Our main work is to
construct barriers relatively to the points $x\in \partial ^{-}\Omega $.

Given $r\in (0,R)$ and $x\in \partial ^{-}\Omega $, let $x^{\ast }$ and $%
\Sigma $ be as defined in (\ref{x*}) and set 
\begin{equation*}
d\left( z\right) =d\left( z,x^{\ast }\right) \text{, }z\in M\text{,}
\end{equation*}%
where $d$ is the Riemannian distance in $M$. Denote by $\rho \left( x\right) 
$ the largest positive number such that 
\begin{equation*}
\exp _{x^{\ast }}:B_{r+\rho \left( x\right) }\left( 0\right) \subset
T_{x^{\ast }}M\longrightarrow \exp _{x^{\ast }}\left( B_{r+\rho \left(
x\right) }\left( 0\right) \right)
\end{equation*}%
is a diffeomorphism (for Hadamard manifold $\rho \left( x\right) =\infty $)
and set 
\begin{eqnarray}
A_{r}^{r+\rho (x)} &:&=\exp _{x^{\ast }}\left( B_{r+\rho \left( x\right)
}\left( 0\right) -B_{r}\left( 0\right) \right)  \label{Ar} \\
&=&\left\{ z\in M-B_{r}\left( x^{\ast }\right) ;\text{ }r\leq d(z)<r+\rho
\left( x\right) \right\} \text{.}  \notag
\end{eqnarray}%
Now, consider the number%
\begin{equation}
\rho :=\min \left\{ \varrho ,\inf \left\{ \rho \left( x\right) ;x\in
\partial ^{-}\Omega \right\} \right\} \text{, }  \label{ro}
\end{equation}%
where $\varrho $ is given in (\ref{ovr}).

In all results of this section, we are considering the following context:

Let $r\in \left( 0,R\right) $ and let $x\in \partial ^{-}\Omega $ be an
arbitrary but fixed point. Let $x^{\ast }$, $\Sigma $ be as defined in (\ref%
{x*}) and let $A_{r}^{r+\rho }$ be defined by (\ref{Ar}) and (\ref{ro}). At $%
z\in A_{r}^{r+\rho }$, consider an orthonormal referential frame $\left\{
E_{i}\right\} $, $i=1,...,n$, where $E_{n}=\nabla d$.

\begin{lemma}
\label{LGeral}Assume $\func{Ric}_{\Omega _{\varrho }}>-\left( n-1\right) 
\mathfrak{H}_{\inf }^{2}$, where $\Omega _{\varrho }$ and $\mathfrak{H}%
_{\inf }$ are given by (\ref{ovr}) and (\ref{infb}), respectively. Given $%
\psi \in C^{2}\left( [r,\infty \right) )$, consider $w\in C^{2}\left(
A_{r}^{r+\rho }\right) $ given by 
\begin{equation}
w\left( z\right) =\left( \psi \circ d\right) \left( z\right) .  \label{w}
\end{equation}%
We have $\mathfrak{M}\left( w\right) \leq 0$ in $A_{r}^{r+\rho }\cap \Omega $%
\emph{\ }if%
\begin{equation}
\psi ^{\prime \prime }+\left( \psi ^{\prime }+\left[ \psi ^{\prime }\right]
^{3}\right) \left( n-1\right) \left\vert \lambda _{r}\right\vert \leq 0,
\label{eg}
\end{equation}%
where $\lambda _{r}$ is given by (\ref{lambda}).
\end{lemma}

\begin{proof}
Straightforward calculus give us that $\mathfrak{M}\left( w\right) \leq 0$
in $A_{r}^{r+\rho }\cap \Omega $\emph{\ }if 
\begin{equation*}
\psi ^{\prime \prime }+\left( \psi ^{\prime }+\left[ \psi ^{\prime }\right]
^{3}\right) \Delta d\leq 0,
\end{equation*}%
where $\Delta $ is the Laplacian in $M$. Since $\func{Ric}_{\Omega _{\varrho
}}>-\left( n-1\right) \mathfrak{H}_{\inf }^{2}$, there is \newline
$0<k<\left\vert \mathfrak{H}_{\inf }\right\vert \leq |\lambda _{r}|$ such
that 
\begin{equation}
\func{Ric}_{\Omega _{\varrho }}\geq -\left( n-1\right) k^{2}.  \label{Ric}
\end{equation}%
Define $f:\left[ 0,\rho \right] \rightarrow \left( 0,+\infty \right) $ by%
\begin{equation*}
f\left( t\right) =k\sinh \left( \coth ^{-1}\left( \frac{\left\vert \mathfrak{%
\lambda }_{r}\right\vert }{k}\right) +kt\right) ,
\end{equation*}%
We have 
\begin{equation}
\frac{f^{\prime \prime }\left( t\right) }{f\left( t\right) }=k^{2},t\in %
\left[ 0,\rho \right] .  \label{f2f}
\end{equation}%
Let $H_{t}$ be the mean curvature of $P_{t}:=\left\{ z\in A_{r}^{r+\rho
};d\left( z\right) =r+t\right\} \subset A_{r}^{r+\rho }$ with respect to the
normal given by $E_{n}=\nabla d$. As $P_{0}=\Sigma $, it follows that

\begin{equation}
H_{0}\geq \lambda _{r}\geq -\frac{f^{\prime }\left( 0\right) }{f\left(
0\right) }.  \label{f1f}
\end{equation}%
Let $\gamma :\left[ 0,\rho \right] \longrightarrow A_{r}^{r+\rho }$ be the
arc length geodesic such that $\gamma \left( 0\right) \in P_{0}$ and $\gamma
^{\prime }\left( t\right) =\nabla d\left( \gamma \left( t\right) \right) $.
We have from (\ref{Ric}) and (\ref{f2f}) that 
\begin{equation}
\func{Ric}_{M}\left( \gamma ^{\prime }\left( t\right) ,\gamma ^{\prime
}\left( t\right) \right) \geq -\left( n-1\right) \frac{f^{\prime \prime
}\left( t\right) }{f\left( t\right) }  \label{c2}
\end{equation}%
for all $t\in \left[ 0,\rho \right] $. Since $\nabla d$ is an extension of $%
\eta $ to $A_{r}^{r+\rho }$ and, in presence of (\ref{f1f}), (\ref{c2}), it
follows from Theorem 5.1 of \cite{KR} that%
\begin{equation*}
-H_{t}\left( \gamma \left( t\right) \right) =\frac{\Delta d\left( \gamma
\left( t\right) \right) }{\left( n-1\right) }\leq \frac{f^{\prime }\left(
t\right) }{f\left( t\right) }.
\end{equation*}%
Then%
\begin{equation*}
\Delta d\leq \left( n-1\right) \frac{f^{\prime }\left( t\right) }{f\left(
t\right) }=\left( n-1\right) k\coth \left( \coth ^{-1}\left( \frac{%
\left\vert \lambda _{r}\right\vert }{k}\right) +kt\right) \leq \left(
n-1\right) \left\vert \mathfrak{\lambda }_{r}\right\vert .
\end{equation*}
\end{proof}

\begin{lemma}
\label{Lcase}Given 
\begin{equation}
0<a<\sqrt{\frac{r}{\left( n-1\right) \left\vert \lambda _{r}\right\vert }}
\label{cota_a}
\end{equation}%
set 
\begin{equation}
s_{0}=\frac{1+\sqrt{1-4\left( n-1\right) ^{2}\lambda _{r}^{2}\left(
a^{2}-r^{2}\right) }}{2\left( n-1\right) \left\vert \lambda _{r}\right\vert }%
.  \label{alfa}
\end{equation}%
Then, the function $w\left( z\right) =a\cosh ^{-1}\left( \frac{d\left(
z\right) }{r}\right) $ satisfies $\mathfrak{M}\left( w\right) \leq 0$ in $%
A_{r}^{r+\mu }\cap \Omega $, where 
\begin{equation}
\mu =\min \left\{ \rho ,s_{0}-r\right\} .  \label{eps}
\end{equation}
\end{lemma}

\begin{proof}
Let $\psi \left( s\right) =\alpha \cosh ^{-1}\left( \frac{s}{r}\right) ,$ $%
s:=d\left( z\right) >r$, where $\alpha >0$ is to be determined. We have%
\begin{equation*}
\psi ^{\prime }\left( s\right) =\frac{\alpha }{\left( s^{2}-r^{2}\right)
^{1/2}}\text{, }\psi ^{\prime \prime }\left( s\right) =\frac{-\alpha s}{%
\left( s^{2}-r^{2}\right) ^{3/2}}
\end{equation*}%
and then, from (\ref{eg}) of Lemma \ref{LGeral}, since $w=\psi \circ d$, $%
\mathfrak{M}\left( w\right) \leq 0$ if%
\begin{equation*}
\frac{-\alpha s}{\left( s^{2}-r^{2}\right) ^{3/2}}+\left[ \frac{\alpha }{%
\left( s^{2}-r^{2}\right) ^{1/2}}+\frac{\alpha ^{3}}{\left(
s^{2}-r^{2}\right) ^{3/2}}\right] \left( n-1\right) \left\vert \lambda
_{r}\right\vert \leq 0,
\end{equation*}%
that is, if 
\begin{equation*}
-s+\left[ s^{2}-r^{2}+\alpha ^{2}\right] \left( n-1\right) \left\vert
\lambda _{r}\right\vert \leq 0,
\end{equation*}%
that is, if 
\begin{equation}
\left( n-1\right) \left\vert \lambda _{r}\right\vert s^{2}-s+\left(
n-1\right) \left\vert \lambda _{r}\right\vert \left( \alpha
^{2}-r^{2}\right) \leq 0.  \label{e-a}
\end{equation}%
In order to get the desirable neighborhood, we need that for $s$ near $r$, $%
s>r$, the inequality (\ref{e-a}) to be strict and that is the case if 
\begin{equation*}
\alpha <\sqrt{\frac{r}{\left( n-1\right) \left\vert \lambda _{r}\right\vert }%
}.
\end{equation*}%
(notice that $\sqrt{\frac{r}{\left( n-1\right) \left\vert \lambda
_{r}\right\vert }}<\frac{\sqrt{1+4\left( n-1\right) ^{2}\lambda _{r}^{2}r^{2}%
}}{2\left( n-1\right) \left\vert \lambda _{r}\right\vert }$). Then, taking $%
\alpha =a$ satisfying (\ref{cota_a}), it follows that for $s\in \lbrack
r,s_{0}]$, where $s_{0}>r$ is given by (\ref{alfa}), the inequality (\ref%
{e-a}) is true, and this concludes the proof.
\end{proof}

\begin{lemma}
\label{boa}Let $a\in (0,\sqrt{\frac{r}{\left( n-1\right) \left\vert \lambda
_{r}\right\vert }})$ and $0<\varepsilon <r\varkappa _{r}$ be given, where $%
\varkappa _{r}$ is given by (\ref{vkapr}). Then, there exist $\delta _{1}>0$
such that $w\left( z\right) =a\cosh ^{-1}\left( \frac{d\left( z\right) }{r}%
\right) $ as defined in Lemma \ref{Lcase} satisfies%
\begin{equation}
w\left( z\right) \geq ar^{-1}\left( r\varkappa _{r}-\varepsilon \right)
^{1/2}d\left( z,x\right) +o\left( d\left( z,x\right) \right) \text{, }z\in
B_{\delta _{1}}\left( x\right) \cap \partial \Omega \text{.}  \label{estboa}
\end{equation}
\end{lemma}

\begin{proof}
Let $g:M\longrightarrow \mathbb{R}$ given by $g\left( z\right) =d\left(
z\right) ^{2}-r^{2}$. Then $\Sigma =g^{-1}\left( 0\right) $. Given $v\in
T_{x}\partial \Omega =T_{x}\Sigma $, $\left\vert v\right\vert =1$, let $Y$
be an extension of $v$ which is tangent to $\Sigma $, that is, $Y\in 
\mathfrak{X}(\Sigma )$ and set $X=\nabla g\left\vert \nabla g\right\vert
^{-1}$. Note that $X$ is an extension to $M$ of the unit normal vector field 
$\eta $, that is, $X|_{\Sigma }=\eta $. As $Y\left( g|_{\Sigma }\right) =0$
and $\nabla g|_{\Sigma }=2r\eta $, setting $\nabla $ the Riemannian
connection of $M$, it follows that on $\Sigma $,%
\begin{eqnarray*}
Hessg(Y,Y) &=&-\left\langle \nabla _{Y}Y,\nabla g\right\rangle
=-\left\langle \nabla _{Y}Y,\eta \right\rangle \left\vert \nabla g\right\vert
\\
&=&-2r\left\langle \nabla _{Y}Y,\eta \right\rangle =-2r\left\langle \left[
\nabla _{Y}Y\right] ^{T}+B\left( Y,Y\right) ,\eta \right\rangle \\
&=&-2r\left\langle B\left( Y,Y\right) ,\eta \right\rangle =-2rII_{\Sigma
}\left( Y\right)
\end{eqnarray*}%
where $II_{\Sigma }$ is the second fundamental form relatively to $\Sigma $
with respect to $\eta $. Therefore, at $x$ we have%
\begin{equation}
Hessg\left( v,v\right) =-2rII_{\Sigma }\left( v\right) .  \label{kns}
\end{equation}%
Let $\alpha :\left[ 0,l\right] \longrightarrow \partial \Omega $, $l>0$, be
an arc length parametrized and simple curve (in the induced metric), such
that $\alpha \left( 0\right) =x$, $\alpha \left( l\right) \neq x$ and $%
\alpha ^{\prime }\left( 0\right) =v$. Let $\sigma $ the arc length parameter
and define%
\begin{equation*}
\xi :\left[ 0,l\right] \longrightarrow \mathbb{R}
\end{equation*}%
by $\xi \left( \sigma \right) =g\left( \alpha \left( \sigma \right) \right) $%
. We have 
\begin{eqnarray*}
\xi ^{\prime }\left( \sigma \right) &=&2d\left( \alpha \left( \sigma \right)
\right) \left\langle \nabla d\left( \alpha \left( \sigma \right) \right)
,\alpha ^{\prime }\left( \sigma \right) \right\rangle = \\
&=&\left\langle \nabla g\left( \alpha \left( \sigma \right) \right) ,\alpha
^{\prime }\left( \sigma \right) \right\rangle
\end{eqnarray*}%
and, relatively to $\partial \Omega $, 
\begin{eqnarray*}
\xi ^{\prime \prime }\left( \sigma \right) &=&\left\langle \nabla _{\alpha
^{\prime }\left( \sigma \right) }\nabla g\left( \alpha \left( \sigma \right)
\right) ,\alpha ^{\prime }\left( \sigma \right) \right\rangle +\left\langle
\nabla g\left( \alpha \left( \sigma \right) \right) ,\nabla _{\alpha
^{\prime }\left( \sigma \right) }\alpha ^{\prime }\left( \sigma \right)
\right\rangle \\
&=&\left\langle \nabla _{\alpha ^{\prime }\left( \sigma \right) }\nabla
g\left( \alpha \left( \sigma \right) \right) ,\alpha ^{\prime }\left( \sigma
\right) \right\rangle + \\
&&+\left\langle \nabla g\left( \alpha \left( \sigma \right) \right) ,\left[
\nabla _{\alpha ^{\prime }\left( \sigma \right) }\alpha ^{\prime }\left(
\sigma \right) \right] ^{T}+B\left( \alpha ^{\prime }\left( \sigma \right)
,\alpha ^{\prime }\left( \sigma \right) \right) \right\rangle \\
&=&\left\langle \nabla _{\alpha ^{\prime }\left( \sigma \right) }\nabla
g\left( \alpha \left( \sigma \right) \right) ,\alpha ^{\prime }\left( \sigma
\right) \right\rangle + \\
&&+\left\vert \nabla g\left( \alpha \left( \sigma \right) \right)
\right\vert \left\langle X\left( \alpha \left( \sigma \right) \right) ,\left[
\nabla _{\alpha ^{\prime }\left( \sigma \right) }\alpha ^{\prime }\left(
\sigma \right) \right] ^{T}+B\left( \alpha ^{\prime }\left( \sigma \right)
,\alpha ^{\prime }\left( \sigma \right) \right) \right\rangle .
\end{eqnarray*}%
As $X$ is normal to $\partial \Omega $ at $x=\alpha \left( 0\right) $, that
is $X\left( x\right) =\eta \left( x\right) =\nu \left( x\right) $, we have 
\begin{eqnarray*}
\xi ^{\prime \prime }\left( 0\right) &=&\left\langle \nabla _{v}\nabla
g\left( x\right) ,v\right\rangle +\left\vert \nabla g\left( x\right)
\right\vert \left\langle \nu \left( x\right) ,B\left( v,v\right)
\right\rangle \\
&=&Hessg\left( v,v\right) +2rII_{\partial \Omega }\left( v\right) ,
\end{eqnarray*}%
where $II_{\partial \Omega }$ is the second fundamental form relatively to $%
\partial \Omega $ with respect to $\nu $ at $x$. Then, from (\ref{kns}), we
obtain%
\begin{equation}
\xi ^{\prime \prime }\left( 0\right) =2r\left( II_{\partial \Omega }\left(
v\right) -II_{\Sigma }\left( v\right) \right) >0,  \label{csi2}
\end{equation}%
being the inequality in (\ref{csi2}) consequence of the fact that $0<r<R$
and from the comparison principle. On the other hand, as $\xi \left(
0\right) =\xi ^{\prime }\left( 0\right) =0$, setting 
\begin{equation*}
2C:=2r\left( II_{\partial \Omega }\left( v\right) -II_{\Sigma }\left(
v\right) \right) ,
\end{equation*}%
we can write%
\begin{equation*}
\xi \left( \sigma \right) =\frac{1}{2}\xi ^{\prime \prime }\left( 0\right)
\sigma ^{2}+\vartheta \left( \sigma \right) =C\sigma ^{2}+\vartheta \left(
\sigma \right)
\end{equation*}%
where 
\begin{equation*}
\underset{\sigma \rightarrow 0}{\lim }\frac{\vartheta \left( \sigma \right) 
}{\sigma ^{2}}=0.
\end{equation*}%
Given $0<\varepsilon <C$, there is $0<\tau \leq l$ such that, for $0<\sigma
\leq \tau $, we have $-\varepsilon \sigma ^{2}<\vartheta \left( \sigma
\right) <\varepsilon \sigma ^{2}$. It follows that 
\begin{eqnarray*}
d\left( \alpha \left( \sigma \right) \right) ^{2} &=&\xi \left( \sigma
\right) +r^{2} \\
&\geq &C\sigma ^{2}+r^{2}-\varepsilon \sigma ^{2}=\left( C-\varepsilon
\right) \sigma ^{2}+r^{2}\geq 0
\end{eqnarray*}%
and then%
\begin{equation}
d\left( \alpha \left( \sigma \right) \right) \geq \sqrt{\left( C-\varepsilon
\right) \sigma ^{2}+r^{2}}\text{, }0\leq \sigma \leq \tau \text{.}
\label{dqc}
\end{equation}%
As $\sigma $ is the arc length parameter, it follows that $\sigma \geq
d\left( \alpha \left( \sigma \right) ,x\right) $. Then, from (\ref{dqc}),
for $0<\sigma \leq \tau $, 
\begin{equation}
a\cosh ^{-1}\left( \frac{d\left( \alpha \left( \sigma \right) \right) }{r}%
\right) \geq a\cosh ^{-1}\left( \frac{\sqrt{\left( C-\varepsilon \right)
d\left( \alpha \left( \sigma \right) ,x\right) ^{2}+r^{2}}}{r}\right) \text{.%
}  \label{dqc2}
\end{equation}%
Let 
\begin{equation*}
\overline{\delta }:=\sup \left\{ d\left( \alpha \left( \sigma \right)
,x\right) \text{, }\sigma \in \left[ 0,\tau \right] \right\} .
\end{equation*}%
Setting $t=d\left( \alpha \left( \sigma \right) ,x\right) $, we have $%
t\rightarrow 0$ when $\sigma \rightarrow 0$. For $t\in \left[ 0,\overline{%
\delta }\right] $, set%
\begin{equation}
h\left( t\right) =a\cosh ^{-1}\left( \frac{\sqrt{\left( C-\varepsilon
\right) t^{2}+r^{2}}}{r}\right) \text{. }  \label{h}
\end{equation}%
Since $h\left( 0\right) =0$ and $h^{\prime }\left( 0\right) =\frac{a}{r}%
\sqrt{C-\varepsilon }$, (\ref{h}) can be rewritten as 
\begin{eqnarray*}
h\left( t\right) &=&h\left( 0\right) +h^{\prime }\left( 0\right) t+\theta
\left( t\right) \\
&=&\frac{a}{r}t\sqrt{C-\varepsilon }+\theta \left( t\right) ,
\end{eqnarray*}%
with 
\begin{equation*}
\underset{t\rightarrow 0}{\lim }\frac{\theta \left( t\right) }{t}=0.
\end{equation*}%
Thus, replacing in (\ref{dqc2}), we have 
\begin{equation}
a\cosh ^{-1}\left( \frac{d\left( \alpha \left( \sigma \right) \right) }{r}%
\right) \geq \left( \frac{a}{r}\sqrt{C-\varepsilon }\right) d\left( \alpha
\left( \sigma \right) ,x\right) +\theta \left( d\left( \alpha \left( \sigma
\right) ,x\right) \right)  \label{dqc3}
\end{equation}%
where 
\begin{equation*}
\underset{\sigma \rightarrow 0}{\lim }\frac{\theta \left( d\left( \alpha
\left( \sigma \right) ,x\right) \right) }{d\left( \alpha \left( \sigma
\right) ,x\right) }=0.
\end{equation*}%
Let $\varkappa \left( x\right) $, $\varkappa _{r}$ as defined in (\ref{vkapx}%
) and (\ref{vkapr}) respectively. We have $0<\varkappa \left( x\right) $
since $0<r<R$ and $\left\{ v\in T_{x}\partial \Omega \text{, }\left\vert
v\right\vert =1\right\} $ is compact. Moreover, as $\partial ^{-}\Omega $ is
compact, it follows that $0<\varkappa _{r}$. Note that $0<r\varkappa
_{r}\leq r\varkappa \left( x\right) \leq C$. Thus, given $0<\varepsilon
<r\varkappa _{r}$, it follows from (\ref{dqc3}) that there is $\delta _{1}>0$
(which does not depend on $x\in \partial ^{-}\Omega $), such that, for all $%
z\in B_{\delta _{1}}\left( x\right) \cap \partial \Omega $, we obtain%
\begin{eqnarray*}
a\cosh ^{-1}\left( \frac{d\left( z\right) }{r}\right) &\geq &\left( \frac{a}{%
r}\sqrt{r\varkappa \left( x\right) -\varepsilon }\right) d\left( z,x\right)
+o\left( d\left( z,x\right) \right) \\
&\geq &\left( \frac{a}{r}\sqrt{r\varkappa _{r}-\varepsilon }\right) d\left(
z,x\right) +o\left( d\left( z,x\right) \right)
\end{eqnarray*}%
that is, 
\begin{equation*}
w\left( z\right) \geq \left( \frac{a}{r}\sqrt{r\varkappa _{r}-\varepsilon }%
\right) d\left( z,x\right) +o\left( d\left( z,x\right) \right) \text{, }z\in
B_{\delta _{1}}\left( x\right) \cap \partial \Omega .
\end{equation*}
\end{proof}

\begin{lemma}
\label{m1}Let $a\in (0,\sqrt{\frac{r}{\left( n-1\right) \left\vert \lambda
_{r}\right\vert }})$. Then%
\begin{equation*}
0<a\sqrt{\frac{\varkappa \left( x\right) }{r}}<\frac{1}{\sqrt{n-1}}.
\end{equation*}%
where $\varkappa \left( x\right) $ is given by (\ref{vkapx}).
\end{lemma}

\begin{proof}
Notice that 
\begin{equation}
\frac{a}{r}\sqrt{r\varkappa \left( x\right) }=\sqrt{\frac{a^{2}}{r}\varkappa
\left( x\right) }\leq \sqrt{\frac{\varkappa \left( x\right) }{\left(
n-1\right) \left\vert \lambda _{r}\right\vert }}=\frac{1}{\sqrt{n-1}}\sqrt{%
\frac{\varkappa \left( x\right) }{\left\vert \lambda _{r}\right\vert }}.
\label{ddiret}
\end{equation}%
On the other hand, there is $v^{\ast }\in T_{x}\partial \Omega $ such that 
\begin{equation*}
\lambda \left( x\right) \leq II_{\Sigma }\left( v^{\ast }\right)
<II_{\partial \Omega }(v^{\ast })<0
\end{equation*}%
and then%
\begin{equation*}
0<-II_{\partial \Omega }(v^{\ast })<-II_{\Sigma }\left( v^{\ast }\right)
\leq \left\vert \lambda \left( x\right) \right\vert .
\end{equation*}%
It follows that%
\begin{equation*}
0<\frac{-II_{\partial \Omega }(v^{\ast })}{\left\vert \lambda \left(
x\right) \right\vert }<\frac{-II_{\Sigma }\left( v^{\ast }\right) }{%
\left\vert \lambda \left( x\right) \right\vert }\leq 1,
\end{equation*}%
that is%
\begin{equation*}
0<\frac{II_{\partial \Omega }(v^{\ast })-II_{\Sigma }\left( v^{\ast }\right) 
}{\left\vert \lambda \left( x\right) \right\vert }<1.
\end{equation*}%
Therefore, as $\varkappa \left( x\right) =\inf \left\{ II_{\partial \Omega
}(v)-II_{\Sigma }\left( v\right) \text{, }v\in T_{x}\partial \Omega \text{, }%
\left\vert v\right\vert =1\right\} ,$%
\begin{equation*}
\sqrt{\frac{\varkappa \left( x\right) }{\left\vert \lambda \left( x\right)
\right\vert }}<1
\end{equation*}%
and then, from (\ref{ddiret}), as $\lambda _{r}\leq \lambda \left( x\right)
<0$, 
\begin{equation*}
a\sqrt{\varkappa _{r}r^{-1}}\leq a\sqrt{\varkappa \left( x\right) r^{-1}}=%
\frac{a}{r}\sqrt{r\varkappa \left( x\right) }<\frac{1}{\sqrt{n-1}}.
\end{equation*}
\end{proof}

\section{Proof of Theorem \protect\ref{MT}}

\begin{proof}
We use the Perron method.

Consider the solution of Perron (\ref{Ps}). In order to show that $u\in
C^{0}\left( \overline{\Omega }\right) $, $u|_{\partial \Omega }=f$, we will
take into account, first, the barriers for the non mean convex points of $%
\partial \Omega $ given by Lemma \ref{boa}.

Let $r\in \left( 0,R\right) $, $a\in (0,\sqrt{\frac{r}{\left( n-1\right)
\left\vert \lambda _{r}\right\vert }})$ and $K\in \lbrack 0,a\sqrt{\varkappa
_{r}r^{-1}})$. Then, there is $0<\varepsilon <r\varkappa _{r}$ such that 
\begin{equation*}
K<\frac{a}{r}\sqrt{r\varkappa _{r}-\varepsilon }\leq \frac{a}{r}\sqrt{%
r\varkappa \left( x\right) -\varepsilon }
\end{equation*}%
for all $x\in \partial ^{-}\Omega $. Then, given $x\in \partial ^{-}\Omega $%
, by the Lemma \ref{boa}, there is $\delta _{1}>0$ such that%
\begin{equation*}
w_{x}\left( z\right) =a\cosh ^{-1}\left( \frac{d\left( z\right) }{r}\right) ,
\end{equation*}%
as in Lemma \ref{Lcase} satisfies%
\begin{eqnarray}
w_{x}\left( z\right) &\geq &ar^{-1}\left( r\varkappa _{r}-\varepsilon
\right) ^{1/2}d\left( z,x\right) +o\left( d\left( z,x\right) \right)  \notag
\\
&\geq &Kd\left( z,x\right) \text{, }z\in B_{\delta _{1}}\left( x\right) \cap
\partial \Omega \text{.}  \label{wx}
\end{eqnarray}
Set%
\begin{equation*}
\delta _{0}=\min \left\{ \delta _{1},\mu \right\} ,
\end{equation*}%
where $\mu $ is given by (\ref{eps}). Then, given $\delta \in (0,\delta
_{0}] $, define%
\begin{equation*}
\epsilon =\inf \left\{ w_{x}\left( z\right) ;\text{ }z\in \partial B_{\delta
}\left( x\right) \cap \partial \Omega ,\forall x\in \partial ^{-}\Omega
\right\} \text{.}
\end{equation*}%
It follows that $\epsilon >0$ since $r<R$. Let 
\begin{equation*}
\omega _{x}^{-}\left( z\right) =f\left( x\right) -w_{x}\left( z\right) \text{%
, }\omega _{x}^{+}\left( z\right) =f\left( x\right) +w_{x}\left( z\right) 
\text{, }x\in \partial ^{-}\Omega \text{, }z\in B_{\delta }\left( x\right)
\cap \overline{\Omega }.
\end{equation*}%
From (\ref{wx}), since 
\begin{equation*}
\left\vert f\left( z\right) -f\left( x\right) \right\vert \leq Kd\left(
z,x\right) \text{, }x\in \partial ^{-}\Omega \text{, }z\in B_{\delta }\left(
x\right) \cap \partial \Omega
\end{equation*}%
we have 
\begin{equation*}
\omega _{x}^{-}\left( z\right) \leq f\left( z\right) \leq \omega
_{x}^{+}\left( z\right) \text{, }x\in \partial ^{-}\Omega \text{, }z\in
B_{\delta }\left( x\right) \cap \partial \Omega
\end{equation*}%
and, since $osc\left( f\right) <\epsilon $, 
\begin{equation*}
\omega _{x}^{-}\left( z\right) <\inf_{\partial \Omega }f\text{, }%
\sup_{\partial \Omega }f<\omega _{x}^{+}\left( z\right) \text{, }x\in
\partial ^{-}\Omega \text{, }z\in \partial B_{\delta }\left( x\right) \cap
\Omega .
\end{equation*}%
Then, setting%
\begin{equation*}
W_{x}^{-}\left( z\right) =\left\{ 
\begin{array}{c}
\max \left\{ \omega _{x}^{-}\left( z\right) ,\inf_{\partial \Omega
}f\right\} ,\text{ if }z\in B_{\delta }\left( x\right) \cap \overline{\Omega 
} \\ 
\inf_{\partial \Omega }f,\text{ if }z\in \overline{\Omega }-B_{\delta
}\left( x\right) \cap \overline{\Omega }%
\end{array}%
\right.
\end{equation*}%
and 
\begin{equation*}
W_{x}^{+}\left( z\right) =\left\{ 
\begin{array}{c}
\min \left\{ \omega _{x}^{+}\left( z\right) ,\sup_{\partial \Omega
}f\right\} ,\text{ if }z\in B_{\delta }\left( x\right) \cap \overline{\Omega 
} \\ 
\sup_{\partial \Omega }f,\text{ if }z\in \overline{\Omega }\backslash
B_{\delta }\left( x\right) \cap \overline{\Omega }%
\end{array}%
\right. ,
\end{equation*}%
we have from Lemma \ref{Lcase} that $W_{x}^{-}$, $W_{x}^{+}$ are subsolution
and supersolution of $\mathfrak{M}$ in $\overline{\Omega }$, respectively, $%
W_{x}^{-}\in S$, $W_{x}^{+}\leq \sup_{\partial \Omega }f$, with $%
W_{x}^{-}\left( x\right) =W_{x}^{+}\left( x\right) =f\left( x\right) $.

At the mean convex points, we proceed as follows: given $x\in \partial
\Omega \backslash \partial ^{-}\Omega $, since $x$ is a mean convex point
and $\partial \Omega $ is of class $C^{2}$, there is a neighborhood $U$ of $%
x $ in $\partial \Omega $ such that $U=\overline{B}\cap \partial \Omega $,
where $B\subset \Omega $ is a mean convex $C^{2}$-domain. The Dirichlet
problem (\ref{PD}) is solvable on $B$ for arbitrary continuous boundary
data. We observe that, here, at the mean convex point, the hypothesis on the
Ricci curvature is necessary only at the points where $\mathfrak{H}\left(
x\right) =0$. We can then choose $g_{x}^{\pm }\in C^{2}\left( B\right) \cap
C^{0}\left( \overline{B}\right) $ satisfying $\mathfrak{M}\left( g_{x}^{\pm
}\right) =0$ in $B$, such that $g_{x}^{\pm }\left( x\right) =f\left(
x\right) $, 
\begin{equation*}
g_{x}^{-}\left( z\right) \leq f\left( z\right) \leq g_{x}^{+}\left( z\right) 
\text{, }z\in U
\end{equation*}%
and 
\begin{equation*}
g_{x}^{-}\left( z\right) <\inf_{\partial \Omega }f\text{, }\sup_{\partial
\Omega }f<g_{x}^{+}\left( z\right) \text{, }z\in \partial B\backslash U.
\end{equation*}%
Then, 
\begin{equation*}
\mathfrak{G}_{x}^{-}\left( z\right) =\left\{ 
\begin{array}{c}
\max \left\{ g_{x}^{-}\left( z\right) ,\inf_{\partial \Omega }f\right\} ,%
\text{ if }z\in \overline{B} \\ 
\inf_{\partial \Omega }f,\text{ if }z\in \overline{\Omega }\backslash 
\overline{B}%
\end{array}%
\right.
\end{equation*}%
and 
\begin{equation*}
\mathfrak{G}_{x}^{+}\left( z\right) =\left\{ 
\begin{array}{c}
\min \left\{ g_{x}^{+}\left( z\right) ,\sup_{\partial \Omega }f\right\} ,%
\text{ if }z\in \overline{B} \\ 
\sup_{\partial \Omega }f,\text{ if }z\in \overline{\Omega }\backslash 
\overline{B}%
\end{array}%
\right.
\end{equation*}%
are subsolution and supersolution of $\mathfrak{M}$ in $\overline{\Omega }$,
respectively, $\mathfrak{G}_{x}^{-}\in S$, $\mathfrak{G}_{x}^{+}\leq
\sup_{\partial \Omega }f$, with $\mathfrak{G}_{x}^{-}\left( x\right) =%
\mathfrak{G}_{x}^{+}\left( x\right) =f\left( x\right) $.

Thus, for each point $x\in \partial \Omega $ we got the barriers and, then,
the solution of Perron (\ref{Ps}) is such that $u\in C^{0}\left( \overline{%
\Omega }\right) $, $u|_{\partial \Omega }=f$.
\end{proof}

\bigskip 

\bigskip 

{\footnotesize *Ari J. Aiolfi: DepMat/Universidade Federal de Santa Maria,
Santa Maria RS/Brazil (ari.aiolfi@mail.ufsm.br);}

{\footnotesize *Giovanni S. Nunes \& Lisandra O. Sauer: IFM/Universidade
Federal de Pelotas, Pelotas RS/Brazil (giovanni.nunes@ufpel.edu.br,
lisandra.sauer@ufpel.edu.br);}

{\footnotesize *Rodrigo B. Soares: IMEF/Universidade Federal de Rio Grande,
Rio Grande RS/Brazil (rodrigosoares@furg.br).}

\end{document}